\documentclass{article}
\usepackage{graphicx}
\usepackage{amsmath}
\usepackage{hyperref}
\usepackage[all,rotate,2cell]{xy}
\usepackage{amsfonts}
\usepackage{amssymb}
\usepackage{tikz}
\usepackage{tikz-cd}
\usepackage{url}
\usepackage{verbatim}

\UseTwocells

\newtheorem{theorem}{Theorem}

\newtheorem{corollary}[theorem]{Corollary}
\newtheorem{definition}[theorem]{Definition}

\newtheorem{lemma}[theorem]{Lemma}
\newtheorem{proposition}[theorem]{Proposition}
\newtheorem{remark}[theorem]{Remark}
\newenvironment{proof}[1][Proof]{\textbf{#1.} }{\ \rule{0.5em}{0.5em}}

\newcommand{\eg}{eg}
\newcommand{\ie}{ie}
\renewcommand{\epsilon}{\varepsilon}

\newcommand{\catname}[1]{\mathcal{#1}}

\newcommand{\baseS}{\catname{S}}             

\newcommand{\Rat}{\mathbb{Q}}                
\newcommand{\Reals}{\mathbb{R}}              

\newcommand{\oftype}{\mathord{:}}

\newcommand{\Id}{\mathop{\mathsf{Id}}}

\newcommand{\fin}{\mathcal{F}}              
\newcommand{\PL}{P_L}                       
\newcommand{\PU}{P_U}                       
\newcommand{\PV}{V}                         
\newcommand{\PCV}{\PV^{c}}                  
\newcommand{\HBC}{\mathsf{HB}_{C}}          
\newcommand{\HeineBorelC}{\mathsf{HB}_C}     

\newcommand{\lint}{\mathop{\underline{\int}}} 
\newcommand{\uint}{\mathop{\overline{\int}}} 
\newcommand{\Leb}[2]{\lambda_{#1 #2}}       
\newcommand{\Unif}[2]{\upsilon_{#1 #2}}       
\newcommand{\tminus}{\mathbin{\dot{-}}}     
\newcommand{\Low}{\mathord{L}}     
\newcommand{\Up}{\mathord{R}}     
\newcommand{\Val}{\mathfrak{V}}            
\newcommand{\Coval}{\mathfrak{C}}          
\newcommand{\slope}[1]{{#1}^{\langle 1 \rangle}} 
\newcommand{\Lintapprox}{\underline{I}}
\newcommand{\Uintapprox}{\overline{I}}

\begin{document}

\title{The Fundamental Theorem of Calculus point-free, with applications to exponentials and logarithms}
\author{Steven Vickers
        \\School of Computer Science, University of Birmingham
        \\ \texttt{s.j.vickers.1@bham.ac.uk}}

\maketitle
\begin{abstract}
	Working in point-free topology under the constraints of geometric logic, we prove the Fundamental Theorem of Calculus, and apply it to prove the usual rules for the derivatives of $x^\alpha$, $\gamma^x$, and $\log_\gamma x$.

  MSC2020-class:
  26E40 
    (Primary)
  26A24, 
  26A36,  
  18F70,  
  18F10 
    (Secondary)
  
\end{abstract}

\section{Introduction} \label{sec:Intro}
This paper is part of a programme of developing real analysis in point-free topology, using the point-based geometric style (in the sense of geometric logic) as in~\cite{NgVic:PtfreeRE}: all constructions are to be geometric, and, consequently, all maps are continuous.
By this means, results of locale theory can be proved without reference to frames of opens.
As far as possible, the aim is to use arguments (and the terminology) already familiar from classical point-set topology.

In~\cite{NgVic:PtfreeRE}, Ng and Vickers gave a point-free account of real exponentiation $\gamma^x$ and logarithms $\log_\gamma x$, together with the expected algebraic laws.
However, they did not cover differentiation, and the motivation for this paper is to fill that gap.
We show, as expected, that $\frac{d}{dx}x^\alpha = \alpha x^{\alpha-1}$, and (with the natural logarithm $\ln x$ defined as $\int_1^x \frac{dt}{t}$) that 
$\frac{d}{dx}\gamma^x = \ln\gamma\, \gamma^x$
and $\frac{d}{dx}\log_\gamma x = 1/x \ln\gamma$.

Much of the working is familiar from classical treatments, albeit with careful choice of the line of argument.
Using the Fundamental Theorem of Calculus (FTC), the differentiability of $\log_\gamma$ is arrived at by studying the integral $\int dt/t$, and then for that of exponentials we can argue with the chain rule.

Our central result, then, is FTC (Theorems~\ref{thm:FTC1} and~\ref{thm:FTC2}).
In fact, this comes out of a relatively simple observation (though I haven't see it before): the derivative of $\int_{x_0}^x g(t) dt$ is obtained from case $x=y$ of the integral $\int_{x}^{y} g d\Unif{x}{y}$, where $\Unif{x}{y}$ is the uniform probability valuation on $[x,y]$.
The main burden of the point-free proof (Theorem~\ref{thm:2Int}) is to show how we can get well behaved 2-sided integrals $\int g d\mu$.
(The 1-sided lower and upper integrals are relatively well understood from~\cite{Integration}.)

\section{Background} \label{sec:BG}
The point-based style of reasoning for point-free spaces, as used for point-free real analysis in~\cite{NgVic:PtfreeRE}, was initiated in~\cite{TopCat}, and has been summarized more recently in~\cite{Vickers:PtfreePtwise}.
In it a \emph{space} is described by a geometric theory (of the points), and a \emph{map} is described by a geometric construction of points of one space out of points of another.
It also allows a natural fibrewise treatment of \emph{bundles}, as maps transforming base points to fibres.

The geometricity ensures that the reasoning applies not only to the global points (models of the theory in a base topos $\baseS$) but also to the generalized points (models in bounded $\baseS$-toposes),
and this circumvents the incompleteness of geometric logic.
This allows a point-based reasoning style like that of ordinary mathematics.
In fact it does better, in that no continuity proofs are needed, and in its treatment of bundles.

For real analysis, the most pervasive differences are an attention to one-sided reals (lower or upper) as well as the Dedekind reals (see Section~\ref{sec:1sidedReals}), and an unexpectedly prominent role for hyperspaces (Section~\ref{sec:Hyperspaces}).

\subsection{Geometric tricks} \label{sec:Tricks}

Here we explain some techniques of geometric reasoning that need a little justification.

The first is that of \textbf{fixing a variable}. It looks like an application of currying, but is in a setting that is not cartesian closed.
Even classically it is not valid, as argumentwise continuity is not enough to prove joint continuity.
It depends on the geometricity.

Suppose we wish to define a map $f\colon X\times Y \to Z$.
We shall commonly \emph{fix} $x\oftype X$, and then define the curried map $f_x\colon Y \to Z$, $y \mapsto f(x,y)$.
However, we are not defining a map from $X$ to a function space $Y\to Z$ -- which need not exist in general.
To ``fix'' $x$ is to work in the topos of sheaves $\baseS X$, where, internally, we define a map $Y\to Z$. Externally, this is a bundle map (over $X$) from $X\times Y$ to $X\times Z$, and this gives us our map from $X\times Y$ to $Z$.

Next, we have a couple of tricks that look like \textbf{case-splitting}.

\begin{theorem} \label{thm:cases}
  Let $X$ be a space, and $X'$ a subspace -- defined by adjoining additional geometric axioms (sequents) to those for $X$.
  
  Let $U$ be an open of $X$, with closed complement $X-U$.
  Then if $U$ and $X-U$ are both contained in $X'$, then $X'$ is the whole of $X$.
\end{theorem}
\begin{proof}
  $U$ and $X-U$ are Boolean complements in the lattice of subspaces of $X$.
  This is well known, but see~\cite{SublocFT} for a discussion tailored to the geometric view.
\end{proof}

Thus in some circumstances, we can reason as if every point of $X$ is in either $U$ or $X-U$, even though that is not true in general.

\begin{theorem} \label{thm:RealPushout}
  The following diagram is a pushout, stable under pullback.
  \[
    \begin{tikzcd}
      1
        \ar[r, "0"]
        \ar[d, "0"']
      & {[0, \infty)}
        \ar[d]
      \\  
      (-\infty, 0]
        \ar[r]
      & \Reals
    \end{tikzcd}
  \]
\end{theorem}
\begin{proof}
  \cite[Section 1.4]{NgVic:PtfreeRE} proves the analogous result for combining $(0,1]$ and $[1,\infty)$. The proof here is essentially the same.
\end{proof}

Thus we can define a map from $\Reals$ by splitting into the non-negative and non-positive cases, and checking that they agree on 0.

\begin{corollary} \label{cor:R2Pushout}
  Let $\leq$ and $\geq$ be the numerical orders on $\Reals$, treated as subspaces of $\Reals^2$.
  $\Reals$ maps into both of them by the diagonal map to $\Reals^2$.
  Then the following diagram is a pushout, stable under pullback.
  \[
    \begin{tikzcd}
      \Reals
        \ar[r]
        \ar[d]
      & \geq
        \ar[d]
      \\  
      \leq
        \ar[r]
      & \Reals^2
    \end{tikzcd}
  \]
\end{corollary}
\begin{proof}
  $\Reals^2$ is isomorphic to itself, by the map $(x,y)\mapsto (u=x+y, v=x-y)$.
  Suppose we are trying to define $f(x,y)$ for the two cases $x\geq y$ and $x \leq y$.
  Fixing $u$, define $g_u(v) = f((u+v)/2, (u-v)/2)$.
  Then the two cases $x\geq y$ and $x\leq y$ correspond to $v\geq 0$ and $v \leq 0$. We can now use Theorem~\ref{thm:RealPushout} to define $g_u$, and we get the required $f(x,y)=g(x+y,x-y)$.
\end{proof}

\subsection{One-sided reals} \label{sec:1sidedReals}

The one-sided reals are lower (approximated from below, corresponding to the topology of lower semi-continuity), or upper (from above, upper semi-continuity).
They are discussed in some detail in~\cite{NgVic:PtfreeRE}, whose notation we shall largely follow.
In the present paper, they play an important role as lower and upper integrals.

We shall most commonly denote spaces of 1-sided reals by a standard notation for 2-sided reals, with a right or left arrow on top for lower or upper reals.
This is motivated by the fact that, unlike the Dedekind reals, they have non-trivial specialization order.
It corresponds to numerical $\leq$ for the lower reals, $\geq$ for the upper.
Then the direction of the arrow in our notation shows the numerical direction of upward specialization order.

A Dedekind (2-sided) real is a pair $x=(\underline{x}, \overline{x})\oftype \overrightarrow{(-\infty,\infty]}
\times \overleftarrow{[-\infty,\infty)}$ satisfying the axioms
\begin{align}
    &\text{disjoint: }
    & 
    \overline{x} < q < \underline{x} &\vdash_{q\oftype\Rat} \bot
    \label{eq:DedAxiom1}
    \\
    &\text{located: }
    &
    q < r &\vdash_{qr\oftype\Rat} 
    q < \underline{x} \vee \overline{x} < r
    \label{eq:DedAxiom2}
\end{align}
We write
$\Low\colon\Reals\to \overrightarrow{(-\infty,\infty]}$
and
$\Up\colon\Reals\to \overleftarrow{[-\infty,\infty)}$
for the maps that extract the lower and upper (or left and right) parts of a Dedekind real.

\subsection{Hyperspaces%
    \protect\footnote{In earlier literature they are known as \emph{powerlocales.} We've changed this in line with our policy of reusing standard topological terminology for the point-free setting.}}
\label{sec:Hyperspaces}
The most conspicuous use of hyperspaces (spaces of subspaces) in the present paper lies in the fact that the interval $[x,y]$ of integration $\int_{x}^{y}$ needs to be geometrically definable from $x$ and $y$.
This is done by treating it as a point in a hyperspace.

However, their roots in point-free analysis go much deeper, and~\cite{CViet} uses them to prove a version of Rolle's Theorem.

Various hyperspaces are known point-free.
The first was the Vietoris $\PV$~\cite{VietLoc}, followed by the lower and upper, $\PL$ and $\PU$.
For more information and references see~\cite{PowerPt}.
They were originally defined in terms of frames, but their geometricity was proved in~\cite{PPExp}.

Finally, the connected Vietoris hyperspace $\PCV$ was defined in~\cite{CViet}, with a proof that the points of $\PCV\Reals$ are equivalent to closed intervals $[x,y]$ ($x\leq y$).
This uses a \emph{Heine-Borel map,} an isomorphism, $\HeineBorelC\colon\mathord{\leq} \to \PCV\Reals$.

If we write $P$ for any of the hyperspaces, then for each point of $PX$ we can define, geometrically, a subspace of $X$. From this we can define a \emph{tautologous bundle} over $PX$.
Following the notational ideas of~\cite{Vickers:PtfreePtwise}, we shall write this as
$\sum_{K\oftype PX} K \to PX$.
The bundle space $\sum_{K\oftype PX} K$ is a subspace of $PX \times X$.

\subsection{Differentiation} \label{sec:Diff}
To define derivatives point-free, we shall follow~\cite{CViet} in using the Carath\'eodory definition, a style developed in some generality in~\cite{BertramGN:DiffCalcGBFR}. For this, $f(x)$ is differentiable if there is a \emph{slope} map $\slope{f}(x,y)$ (continuous, as are all maps point-free) satisfying
$f(y)-f(x) = (y-x) \slope{f}(x,y)$.
Then the derivative $f'(x)$ is $\slope{f}(x,x)$.
Note that $f'$ is necessarily continuous: differentiable here means \emph{continuously} differentiable ($C^1$).

By calculations that are completely familiar, we see that the class of differentiable maps contains constant maps and the identity, and is closed under various operations.

\[ \begin{split}
  \slope{(x\mapsto c)}(x,y) &= 0 \\
  \slope{\Id}(x,y) &= 1 \\
  \slope{(f+g)}(x,y) &= \slope{f}(x,y) + \slope{g}(x,y) \\
  \slope{(fg)}(x,y) &= \slope{f}(x,y)g(y) + f(x)\slope{g}(x,y)\\
  \slope{\left( \frac{1}{f} \right)}(x,y) &= -\frac{\slope{f}(x,y)}{f(x)f(y)} \quad \text{(if every $f(x)$ non-zero)} \\
  \slope{(f\circ g)}(x,y) &= \slope{f}(g(x),g(y)) \slope{g}(x,y) \quad \text{\emph{(Chain Rule)}}
\end{split} \]

\begin{lemma} \label{lem:invdiff}
  Suppose $f$ and $g$ are mutually inverse maps (between open subspaces of $\Reals$), and $f$ is differentiable with $\slope{f}(x',y')\ne 0$ for all $x',y'$.
  Then $g$ is differentiable, with $\slope{g}(x,y) = 1/\slope{f}(g(x),g(y))$ and $g'(x) = 1/f'(g(x))$.
\end{lemma}
\begin{proof}
  From $f(y')-f(x') = (y'-x')\slope{f}(x',y')$ we deduce
  $y-x = (g(y)-g(x))\slope{f}(g(x),g(y))$.
\end{proof}

In all these examples the slope maps can be defined explicitly.
The value of the FTC is that it enables us to define slope maps as integrals, once we know what to expect for the derivative.

\subsection{Preliminaries on integration} \label{sec:Int}

Let us first recall some of the point-free theory of integration.
We follow the account of Vickers~\cite{Integration}, with some reference also to that of Coquand and Spitter~\cite{CoqSpit:IntVal}.
Vickers's account is very general, with integration over an arbitrary locale $X$.

They make careful use of one-sided reals, lower (approximated from below, topology of lower semi-continuity) or upper (from above, upper semi-continuity).

\subsubsection*{Lower integrals} \label{sec:LowerInt}

For a lower integral $\lint_X f d\mu$, both integrand and measure are non-negative lower reals.
More specifically, the measure is a \emph{valuation,} a Scott continuous map $\mu\colon \Omega X \to \overrightarrow{[0,\infty]}$
with $\mu\emptyset=0$ and satisfying the modular law
$\mu U + \mu V = \mu(U\vee V) + \mu(U\wedge V)$.
The valuations on $X$ form a space $\Val X$,
and $\Val$ is a functor -- indeed, a monad.

$\mu$ is a \emph{probability valuation} if, in addition, $\mu X=1$, and, more generally, it is \emph{finite} if $\mu X$ is a Dedekind real.
(More carefully, a finite valuation is a pair $(\mu,a)$ where $\mu$ is a valuation, and $a$ is a Dedekind real for which $\Low a = \mu X$.)

The lower integral $\lint_X f d\mu$ is defined in two steps.
First, it is reformulated as ${\lint}_{\overrightarrow{[0,\infty]}} \Id d\Val f(\mu)$,
so that we may assume without loss of generality that $X=\overrightarrow{[0,\infty]}$ and $f$ is the identity;
and then we define
${\lint}_{\overrightarrow{[0,\infty]}} \Id d\mu$
where $\mu$ is now a valuation on $\overrightarrow{[0,\infty]}$.
Combining them, we get $\lint_X f d\mu$ as the supremum of expressions
\begin{equation*} \label{eq:LowInt}
  \Lintapprox(r_i)_i = \sum_{i=1}^n (r_i - r_{i-1})\mu f^\ast\overrightarrow{(r_i,\infty]}
\end{equation*}
over rational sequences $0=r_0 < \cdots < r_n$ ($n \geq 1$).
In effect, this is defining a Choquet integral. See Figure~\ref{fig:lower}.

\begin{figure}
  \begin{center}
    \includegraphics[width=.65\textwidth]{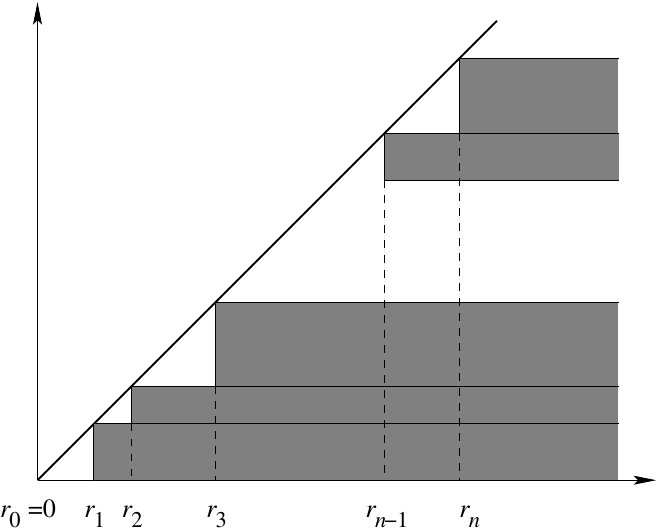}
    \caption{Illustration of lower integral.}
    \label{fig:lower}
 \end{center}
\end{figure}

\begin{remark} \label{rem:LowInt}
  Here are some properties not mentioned in~\cite{Integration}.
  
  First, it makes no difference if the sequence $(r_i)_{i=0}^n$ increases \emph{non-}strictly, as the repeated points make zero contribution.
  
  Second, refining the sequence (adding extra points) does not decrease the value of $\Lintapprox$,
  because if $r_{i-1} < r' < r_i$ then
  \[ \begin{split}
    (r_i - r_{i-1})\mu f^\ast\overrightarrow{(r_i,\infty]}
      &= (r_i - r')\mu f^\ast\overrightarrow{(r_i,\infty]}
        + (r' - r_{i-1}) f^\ast\mu\overrightarrow{(r_i,\infty]}
    \\
      &\leq (r_i - r')\mu f^\ast\overrightarrow{(r_i,\infty]}
      + (r' - r_{i-1})\mu f^\ast\overrightarrow{(r',\infty]}
    \text{.}
  \end{split} \]

  Third, by finding a common denominator for the rationals $r_i$ and refining, we can assume without loss of generality that the sequence is in arithmetic progression, $r_i=ir_n/n$.
\end{remark}

\subsubsection*{Upper integrals} \label{sec:UpperInt}

For an upper integral $\uint_X f d\nu$, both integrand and measure are non-negative upper reals.
This time, the measure is a \emph{co}valuation,
a Scott continuous map $\nu\colon \Omega X \to \overleftarrow{[0,\infty]}$,
which gives the measure on \emph{closed} subspaces:
$\nu U$ is the mass of $X-U$.
(Note that the specialization order on upper reals is the reverse of the numerical order, so $\nu$ is numerically contravariant.)
Again, it satisfies the modular law, but this time with $\nu X=0$.
The covaluations form a space $\Coval X$, and again $\Coval$ is a functor.

$\nu$ is a \emph{probability covaluation} if, in addition, $\nu \emptyset=1$, and \emph{finite} if $\nu \emptyset$ is a Dedekind real.
We also say that $\nu$ is \emph{bounded} if $\nu\emptyset < q$ for some rational $q$ -- so finite implies bounded.

Note that there is an equivalence between finite valuations and finite covaluations.
If $\mu$ is a finite valuation, then its \emph{complement} covaluation $\neg\mu$ is defined by $\neg\mu U = \mu X - \mu U$,
and similarly for a finite covaluation $\nu$.

Again, the upper integral ${\uint}_X f d\nu$ is defined~\cite{Integration} in two steps, the first more complicated this time.
We assume $X$ is compact, and $f\colon X \to \overleftarrow{[0,\infty)}$ --
that is to say, $\infty$ is excluded from the codomain of $f$, so every $f(x)$ is finite.
However, the following Lemma allows us to use compactness of $X$ to find a uniform bound.
It is a one-sided version of results in~\cite[Section 7.1]{CViet}, which show that a positive point of $\PV \Reals$ has both a $\sup$ and an $\inf$.

\begin{lemma} \label{lem:compSup}
  There is a map
  \[
    \sup \colon \PU \overleftarrow{[0,\infty)} \to
      \overleftarrow{[0,\infty)}
  \]
  such that $\sup K$ is the least upper bound of the points in $K$.
\end{lemma}
\begin{proof}
  As subspace of $\overleftarrow{[-\infty,\infty)}$, $\overleftarrow{[0,\infty)}$ is the closed complement of $\overleftarrow{[-\infty,0)}$. It follows that its points are the rounded, inhabited, upsets of the rationals that don't contain 0: in other words, the rounded ideals of $(Q_{+}, >)$ where $Q_{+}$ is the set of positive rationals.
  We can therefore use the methods of~\cite{Infosys}.
  This is aided by the fact that $(Q_{+}, >)$ is an \emph{R-structure,} meaning that, for each $q\in Q_{+}$, the set of rationals $r>q$ is a rounded ideal (with respect to $>$, ie a rounded filter numerically).
  This allows us to use the ``Lifting Lemma'' of~\cite[1.29]{NgVic:PtfreeRE} to define maps out of
  a rounded ideal completion from their restrictions to the base order.
  
  From~\cite{Infosys}, we see that $\PU\overleftarrow{[0,\infty)}$ is the rounded ideal completion of $(\fin Q_{+}, >_U)$, where $\fin$ is the Kuratowski finite powerset, and $S >_U T$ if for every $t\in T$, there is some $s\in S$ with $s>t$.
  We define $\sup$ by its restriction
  $\max\colon \fin Q_{+} \to \overleftarrow{[0,\infty)}$,
  $\max S$ being the actual maximum element if $S$ is inhabited, or 0 if empty.
  (This case splitting is legitimate, as emptiness of finite sets is decidable.)
  
  The two conditions of the Lifting Lemma are now easily checked. Monotonicity says that if $S>_U T$ then $\max S \geq \max T$; and Continuity says that $\max T$ equals $\inf_{S>_U T} \max S$.

  The final part uses arguments similar to those of~\cite[Section 7.1]{CViet}.
\end{proof}

Returning to our definition of the upper integral, $X$ itself can be considered a point in $\PU X$, and then $\PU f(X)$ is a point in
$\PU \overleftarrow{[0,\infty)}$,
and, by Lemma~\ref{lem:compSup}, has a $\sup$ in $\overleftarrow{[0,\infty)}$. Hence there is some rational $r$ such that $X\leq f^{\ast}\overleftarrow{[0,r)}$.

Now ${\uint}_X f d\nu$ is defined as the infimum of the expressions
\begin{equation*} \label{eq:UpInt}
  \Uintapprox(r_i)_i = \sum_{i=1}^n (r_i - r_{i-1})\nu f^\ast\overleftarrow{[0,r_{i-1})}
\end{equation*}
over rational sequences $0 = r_0 < \cdots < r_n$ with $n\geq 1$ and $X \leq f^\ast\overleftarrow{[0,r_n)}$.
See Figure~\ref{fig:upper}.

\begin{figure}
    \begin{center}
        \includegraphics[width=.65\textwidth]
          {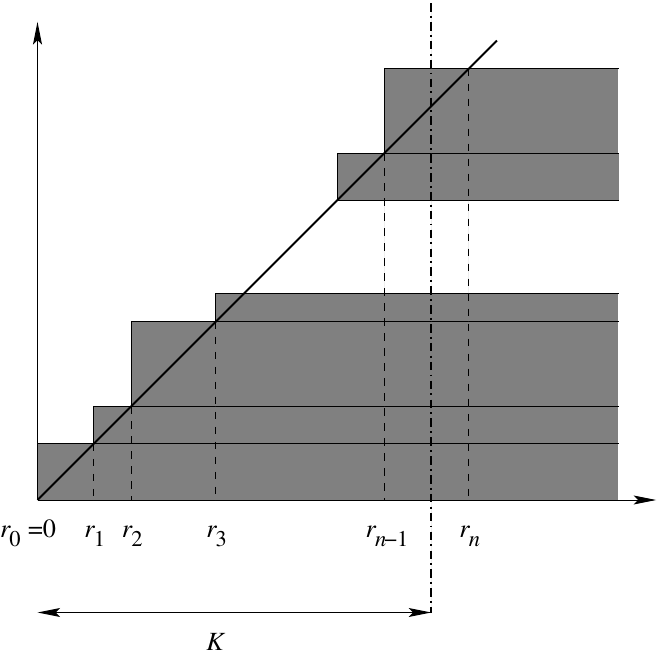}
        \caption{Illustration of upper integral.
            $K = \sup(\PU f(X))$.}
        \label{fig:upper}
    \end{center}
\end{figure}

\begin{remark} \label{rem:UpInt}
  If $\nu$ is bounded, then we have properties analogous to those of Remark~\ref{rem:LowInt} -- except that refining the sequence cannot \emph{in}crease the value of $\Uintapprox$.
  
  Also, if $X \leq f^\ast\overleftarrow{[0,r_{n-1})}$, then we can omit $r_n$ without affecting the value of $\Uintapprox$, because $\nu f^\ast \overleftarrow{[0,r_n)} = \nu X = 0$. 

  (Without boundedness we run into the problem that, for upper reals, $0\cdot\infty = \inf_{0<q\oftype\Rat} q\cdot\infty = \infty$.)
\end{remark}

\section{Two-sided integrals} \label{sec:2Int}
Once we have lower and upper integrals, which are respectively lower and upper reals, it is natural to ask when they can be put together to make a 2-sided real.
\cite{Integration} shows how to do this in the particular case that corresponds to Riemann integration, with $X$ the real interval $[0,1]$ equipped with the Lebesgue valuation.

We now show how to generalize this.
$X$ will be any space that is compact (which enables us to find the bound required for an upper integral); $\mu$ will be any finite valuation on it (with the upper integral provided using the covaluation $\neg\mu$); and the integrand $f$ will be valued in the non-negative Dedekind reals $[0,\infty)$.
Then we have a pair of lower and upper reals,
\begin{equation} \label{eq:2Int}
  {\int}_X f d\mu
    = \left({\lint}_X \Low(f) d\mu, {\uint}_X \Up(f) d(\neg \mu)\right)
  \text{,}
\end{equation}
where $\Low$ and $\Up$ provide the lower and upper parts of a Dedekind real.

In Theorem~\ref{thm:2Int} we shall show that this is a Dedekind real.
Once this is done, we can generalize to signed integrands, $f\colon X \to \Reals$, by decomposing $f$ as $f_+ - f_-$,
where $f_+(x)= \max(0,f(x))$ and $f_-(x)= \max(0,-f(x))$, and then defining
\begin{equation*} \label{eq:2Intsigned}
  {\int}_X f d\mu = {\int}_X f_+ d\mu - {\int}_X f_- d\mu
  \text{.}
\end{equation*}

\begin{theorem} \label{thm:2Int}
  Let $X$ be compact, $\mu$ a finite valuation on $X$, and $f\colon X \to [0,\infty)$.
  Then ${\int}_X f d\mu$ (equation~\eqref{eq:2Int}) is a Dedekind section.
\end{theorem}
\begin{proof}
  First, it is clear that $\Val(\Low(f))(\mu)$ and $\Coval(\Up(f))(\neg\mu)$ are both finite,
  with the same total mass $\mu X$.
  
  We show that ${\int}_X f d\mu$ as defined in~\eqref{eq:2Int} satisfies the axioms~\eqref{eq:DedAxiom1} and~\eqref{eq:DedAxiom2}.

  For disjointness, suppose we have
  sequences $(r_i)_i$ and $(r'_{i'})_{i'}$ such that
  $\Uintapprox(r_i)_i < q$ (with $\neg\mu$ for $\nu$))
  and $q < \Lintapprox(r'_{i'})_{i'}$.
  By refining, we can assume they use the same sequence $(r_i)_i$, so
  \[
    \sum_{i=1}^n (r_i - r_{i-1})(\mu X - \mu f^\ast[0,r_{i-1}))
    < q <
      \sum_{i=1}^n (r_i - r_{i-1})\mu f^\ast(r_i,\infty)
    \text{.}
  \]
  Then we can find $q = \sum_{i=1}^n q'_i = \sum_{i=1}^n q''_i$
  with
  \[
    (r_i - r_{i-1})(\mu X - \mu f^\ast[0,r_{i-1})) < q'_i
    \text{ and }
    q''_i < (r_i - r_{i-1})\mu f^\ast(r_i,\infty)
    \text{.}
  \]
  We cannot have $q'_i > q''_i$ for all $i$, so there must be $i$ with $q'_i \leq q''_i$. Then
  \[
    \mu X - \mu f^\ast[0,r_{i-1})
      < q'_i/(r_i - r_{i-1})
      < \mu f^\ast(r_i,\infty)
    \text{,}
  \]
  which gives a contradiction from
  \[
    \mu X < \mu f^\ast[0,r_{i-1}) + \mu f^\ast(r_i,\infty)
      = \mu f^\ast([0,r_{i-1}) \vee (r_i,\infty))
      \leq \mu X
    \text{.}
  \]
  
  For locatedness, given $q<r$, we need to find a sequence $(r_i)_i$ such that either $q<\Lintapprox(r_i)_i$ or $\Uintapprox(r_i)_i < r$. 
  Our strategy will be to find a sequence for which
  $\Lintapprox(r_i)_i$ and $\Uintapprox(r_i)_i$ are sufficiently close.
  Referring to Figures~\ref{fig:lower} and~\ref{fig:upper}, their difference is the total size of the squares along the diagonal, and we wish to make these small.
  
  According to Remark~\ref{rem:UpInt}, we can assume the sequence is in arithmetic progression, so
  \[
    \Uintapprox(r_i)_i = \sum_{i=1}^n (r_i - r_{i-1})(\mu X - \mu f^\ast[0,r_{i-1}))
      = r_n \mu X - r_1 \sum_{i=1}^n \mu f^\ast[0,r_{i-1})
    \text{.}
  \]

  If we write $a = r_1 \sum_{i=1}^n \mu f^\ast[0,r_{i-1})$,
  then we are trying to find a sequence sufficiently refined so that $\Lintapprox(r_i)_i + a$ is large enough to be close to $r_n \mu X$.
  
  First, let us choose $r_n$ so that $X\leq f^\ast[0,r_n)$, in other words $K < r_n$ in Figure~\ref{fig:upper}.
  (It remains to choose $n$ so that $r_1=r_n/n$ is sufficiently small.)
  
  We calculate
  \[ \begin{split}
	\Lintapprox(r_i)_i + a
	&= r_1 \left( \sum_{i=1}^{n-1} \mu f^\ast(r_i, \infty)
	+ \sum_{i=1}^n \mu f^\ast [0,r_{i-1}) \right)\\
	& \geq r_1 \sum_{i=2}^{n-1}
	\left( \mu f^\ast(r_{i-1}, \infty) + \mu f^\ast[0,r_i) \right) \\
	& = r_1 \sum_{i=2}^{n-1}
	\left( \mu X + \mu f^\ast(r_{i-1},r_i)\right) \\
	& \geq (r_n - 2r_1) \mu X
  \end{split} \]
  
  Now let us find $q' < q'' < r_n \mu X < r'$ with $r'-q' \leq r-q$, and choose $n$ so that $2r_1 \mu X < q''-q'$.
  Then $q' < \Lintapprox(r_i)_i + a$,
  so we can find $q' = s+t$ such that $s < \Lintapprox(r_i)_i$ and $t < a$.
  
  If $q\leq s$ then $q<\Lintapprox(r_i)_i$.

  Otherwise, $s < q$, we have
  \[
    r_n \mu X - r < r'-r \leq q'-q = s+t-q < t < a
    \text{,}
  \]
  so $\Uintapprox(r_i)_i = r_n \mu X -a < r$.
\end{proof}

\subsection{Valuations on real intervals} \label{sec:ValRealInt}

We shall be interested in the case where $X$ is a compact real interval $[x,y]$ (with $x \leq y$ both Dedekind),
and $\mu$ is either Lebesgue valuation $\Leb{x}{y}$ or the uniform probability valuation $\Unif{x}{y}$.
These are all finite, and so we have corresponding covaluations.

Note that all our constructions can be regarded as continuous maps.
For example, the one that takes a pair $(x,y)$ of reals, with $x \leq y$, and returns a compact subspace $[x,y]$ of $\Reals$, is the \emph{Heine-Borel map} $\HBC$ of~\cite{CViet}, from $\leq$ as subspace of $\Reals^2$ to the Vietoris hyperspace (powerlocale) $\PV\Reals$.
The hyperspace comes equipped with a tautologous bundle, which allows us to consider \eg\ $\Leb{x}{y}$ as a point of $\Val[x,y]$, so $(x,y)\mapsto \left([x,y], \Leb{x}{y}\right)$
is a continuous map.
This allows us to say that an integral is continuous in its bounds.

We must avoid the temptation to think every $[x,y]$ is homeomorphic to $[0,1]$, and that it suffices to take $x=0$ and $y=1$.
This fails in the important case $x=y$, and defining the homeomorphism goes beyond the scope of Theorem~\ref{thm:cases}.
We shall need a uniform way to cover both cases, $x\ne y$ and $x=y$.

\begin{lemma} \label{lem:RVal}
  To define a valuation $\mu$ on $\Reals$, it suffices to define $\mu(a,a')\oftype \overrightarrow{[0,\infty]}$ for all rationals $a<a'$, so that the following conditions hold.
  \begin{enumerate}
  \item
    \[
      \mu(a,a')=
        \sup \{ \mu(b, b') \mid a < b < b' < a' \}
      \text{,}
    \]
  \item
    If $a < b < a' < b'$ then
    \[
      \mu(a, a') + \mu(b, b') = \mu(a, b') + \mu(b, a')
      \text{.}
    \]
    (Note that, in the presence of (1), this then also holds for $a\leq b < a' \leq b'$.)
  \end{enumerate}

  In addition, for a valuation on $[x,y]$, we need that $\mu(a,a') = 0$ iff $a'<x$ or $y<a$.
\end{lemma}
\begin{proof}
  We use~\cite[Proposition 4.1]{Integration}. It shows how the Scott continuity of a valuation allows it to be defined on a distributive lattice $L$ of generators, of which every open is a directed join.
  It suffices to define $\mu$ as a monotone map from $L$ to $\overrightarrow{[0,\infty]}$ that satisfies $\mu0 = 0$ and the modular law, and preserves certain specified directed joins.
    
  For $\Reals$ it is clear that every open is a directed join of finite joins of rational open intervals, $\bigvee_1^n (a_i,a'_i)$ with $a_i<a'_i$.
  If two of the intervals overlap then we can replace them by their envelope, and so we can reduce to the case where the intervals are disjoint.
  We can also rearrange to get $a'_i \leq a_{i+1}$.
  Now, by the modular law, we must have
  \begin{equation} \label{eq:muJoin}
    \mu \left( \bigvee_1^n (a_i,a'_i) \right)
      = \sum_i \mu (a_i, a'_i)
    \text{,}
  \end{equation}
  so it suffices to define $\mu$ on non-empty rational open intervals.

  For us, then, $L$ is the distributive lattice of finite, disjoint sets of rational open intervals together with an adjoined top.
  The specified directed joins are 
  $\bigvee_i (a_i,a'_i)$ as join of elements $\bigvee_i (b_i,b'_i)$ such that $a_i<b_i<b'_i<a'_i$, and top as join of all elements $\bigvee_i (a_i,a'_i)$.
  Preservation of those, together with monotonicity, are easily derived from condition (1).
  
  For modularity, consider $u=\bigvee_i(a_i,a'_i)$ and $v=\bigvee_j(b_j,b'_j)$.
  We have to consider the process by which $u\vee v$ is reduced to the canonical form.
  We reduce from $(u,v)$ to some $(u', v')$ with fewer disjuncts, and use induction.
  Assuming (without loss of generality) that $a_1\leq b_1$, there are three cases.
  \begin{itemize}
  \item
    If $a'_1 \leq b_1$, then we can detach $(a_1,a'_1)$ from $u$, giving $u'$, and leave $v'=v$.
    Then $\mu (u\vee v)= \mu(a_1,a'_1)+\mu(u'\vee v')$
    and $\mu (u\wedge v)= \mu(u'\wedge v')$.
  \item
    If $b'_1 \leq a'_1$, then we can detach $(b_1,b'_1)$ from $v$, giving $v'$, and leave $u'=u$.
    Then $\mu (u\vee v)= \mu(u'\vee v')$
    and $\mu (u\wedge v)= \mu(b_1, b'_1) +\mu(u'\wedge v')$.
  \item
    If $b_1 < a'_1 < b'_1$, then we remove $(a_1, a'_1)$ from $u$, giving $u'$, and in $v$ we replace $(b_1,b'_1)$ by $(a_1, b'_1)$, giving $v'$.
    Then $\mu (u\vee v)= \mu(u'\vee v')$
    and $\mu (u\wedge v)= \mu(b_1, a'_1) +\mu(u'\wedge v')$. Hence, writing $v''$ for $v$ with $(b_1,b'_1)$ removed,
    \[
      \begin{split}
        \mu(u\vee v) &+ \mu(u\wedge v)
          = \mu(u'\vee v') + \mu(b_1, a'_1) + \mu(u'\wedge v') \\
        &= \mu(b_1, a'_1) + \mu u' + \mu v' \\
        &= \mu(b_1, a'_1) + \mu u' + \mu(a_1, b'_1) + \mu v'' \\
        &= \mu(a_1, a'_1) + \mu(b_1, b'_1) + \mu u' + \mu v''
          = \mu u + \mu v \text{.}
      \end{split}
    \]
  \end{itemize}

  For the final part, $L$ is modified to take on the axioms that $(a,a')\leq \bot$ if $a'<x$ or $y<a$, and the rest is straightforward.
\end{proof}

\subsection{The Lebesgue valuation \texorpdfstring{$\Leb{x}{y}$}{lambda\_xy}
on \texorpdfstring{$[x,y]$}{[x,y]}}
\label{sec:Lebesgue}

\begin{definition} \label{def:Leb}
  Let $x\leq y$ be Dedekind reals. Then the Lebesgue valuation $\Leb{x}{y}$ is defined by
  \[
    \Leb{x}{y}(a,b) = \Low \left(\min(b,y)\tminus\max(a,x)\right)
    \text{.}
  \]
\end{definition}

Here, $\tminus$ is the \emph{truncated minus,} $y\tminus x = \max(0, y-x)$.

It is straightforward to check the conditions of Lemma~\ref{lem:RVal}.
Of course, we can also define the Lebesgue valuation $\Leb{}{}$ on $\Reals$ by $\Leb{}{}(a,b) = \Low(b-a)$.

\begin{proposition} \label{prop:Leb}
	If $x<y$, this defines the same Lebesgue valuation as in~\cite{Integration}.
\end{proposition}
\begin{proof}
  The formula given there is for $[0,1]$, to be generalized to $[x,y]$ ``by scaling''. This amounts to
  \[
    \Leb{x}{y}(U) = \sup \left\{ \Low \left(\sum_i(r_i- q_i)\right)\right\}
  \]
  where the $(q_i,r_i)$s are disjoint rational open intervals included in $U$.
  For $U$ of the form $(a,b)$, that reduces to $b-a$ if the interval $(a,b)$ is contained in $[x,y]$.
  More generally we must find the intersection of $(a,b)$ with $[x,y]$ and take the length of that (or, rather, its lower part), and that is what is given by Definition~\ref{def:Leb}.
\end{proof}

Note that $\Leb{x}{y}([x,y])=\Low(y-x)$, including the case $x=y$. Also, if $(a,b)$ is disjoint from $[x,y]$, \ie\ if $b\leq x$ or $y\leq a$, then $\Leb{x}{y}(a,b)=0$.

If $f\colon [x,y] \to \Reals$, then from Theorem~\ref{thm:2Int} we have a 2-sided integral ${\int}_{[x,y]} f d\Leb{x}{y}$.
For $[0,1]$, \cite{Integration} already showed the existence of this and that it equals the Riemann integral $\int_0^1 f(t)dt$; the proof extends to the case of $[x,y]$ for $x<y$, and clearly it also holds for $x=y$.

We shall write $\int_x^y f(t)dt$ for ${\int}_{[x,y]} f d\Leb{x}{y}$ generally for $x\leq y$.
Our aim now is to extend that to the case $y\leq x$, and prove the interval additivity equation:   \begin{equation} \label{eq:intsum}
    \int_x^z f(t) dt = \int_x^y f(t) dt + \int_y^z f(t) dt
    \text{.}
\end{equation}

For $x\leq y$, let us write $i_{xy}\colon [x,y] \to \Reals$ for the inclusion.

\begin{proposition} \label{prop:LebIntervals}
  If $x\leq y\leq z$ are Dedekind reals,
  then
  \[
    \Val i_{xz}(\Leb{x}{z}) = \Val i_{xy}(\Leb{x}{y}) + \Val i_{yz}(\Leb{y}{z})
    \text{,}
  \]
  and
  \[
    \Coval i_{xz}(\neg\Leb{x}{z})
      = \Coval i_{xy}(\neg\Leb{x}{y}) + \Coval i_{yz}(\neg\Leb{y}{z})
    \text{.}
  \]
\end{proposition}
\begin{proof}
  For the first part we require, for all rational $a<b$, that
  \[
    \Leb{x}{z}(a,b) = \Leb{x}{y}(a,b) + \Leb{y}{z}(a,b)
    \text{,}
  \]
  in other words, that
  \[
    (\min(b,z)\tminus \max(a,x))
      = (\min(b,y)\tminus \max(a,x)) + (\min(b,z)\tminus \max(a,y))
    \text{.}
  \]
  Fixing $a$ and $b$, the collection of triples $(x,y,z)$ satisfying this forms a subspace of $\Reals^3$,
  so it suffices to check on the open subspace for which $a<y<b$, and its closed complement for which $b\leq y$ or $y\leq a$.
  
  If $a<y<b$, then
  \[
    \text{RHS } = (y-\max(a,x)) + (\min(b,z)-y) = \text{ LHS.}
  \]
  
  If $b\leq y$, then
  \[
    \text{RHS } = b \tminus \max(a,x) = \text{ LHS,}
  \]
  and similarly if $y\leq a$.
  
  The second part is immediate from the first. 
\end{proof}

\begin{corollary} \label{cor:RiemannAddIntervals}
  If $x\leq y \leq z$ then the interval additivity equation~\ref{eq:intsum} holds.
\end{corollary}
\begin{proof}
  We deduce from Proposition~\ref{prop:LebIntervals} that if $f\colon [x,z] \to \overrightarrow{[0,\infty]}$, then
  \begin{equation} \label{eq:lintsum}
    {\lint}_{[x,z]} f\circ i_{xz} d\Leb{x}{z}
      = {\lint}_{[x,y]} f\circ i_{xy} d\Leb{x}{y}
        + {\lint}_{[y,z]} f\circ i_{yz} d\Leb{y}{z}
    \text{,}
  \end{equation}
  and if $f\colon [x,z] \to \overleftarrow{[0,\infty)}$, then
  \begin{equation} \label{eq:uintsum}
    {\uint}_{[x,z]} f\circ i_{xz} d\neg\Leb{x}{z}
      = {\uint}_{[x,y]} f\circ i_{xy} d\neg\Leb{x}{y}
        + {\uint}_{[y,z]} f\circ i_{yz} d\neg\Leb{y}{z}
    \text{.}
  \end{equation}

  The result follows by combining equations~\eqref{eq:lintsum} and~\eqref{eq:uintsum}.
\end{proof}

We complete this section by extending to integration over negative intervals:

\begin{theorem} \label{thm:RiemannNegInt}
  We can extend the definition of $\int_{x}^{y} f(t) dt$ to arbitrary $x,y$, in such a way that
  $\int_{x}^{y} f(t) dt = -\int_{y}^{x} f(t) dt$
  and equation~\eqref{eq:intsum} holds for all $x,y,z$.
\end{theorem}
\begin{proof}
  If $y\leq x$ then we must define $\int_x^y f(t)dt = -\int_y^x f(t) dt$.
  On the overlap case $x=y$, then these agree, giving 0.
  We can now apply Corollary~\ref{cor:R2Pushout}.

  Now we check that equation~\eqref{eq:intsum} holds for arbitrary $x,y,z$.

  First, if it holds for a triple $(x,y,z)$, then it is easily checked that it holds for all permutations of it. 

  Next, if $x$, $y$ and $z$ are rational, then the equation holds: for, by decidability of order on the rationals, there is some permutation that puts them in order.

  Finally, the subspace of $\Reals^3$ comprising those triples for which the equation holds is closed. Since it includes the image of the dense map from $\Rat^3$, it is the whole of $\Reals^3$.
\end{proof}

\subsection{The uniform probability valuation
  \texorpdfstring{$\Unif{x}{y}$}{upsilon\_xy}
  on \texorpdfstring{$[x,y]$}{[x,y]}}
\label{sec:Uniform}
\begin{definition} \label{def:Unif}
    Let $x\leq y$ be Dedekind reals.
    Then the uniform valuation $\Unif{x}{y}$ is defined by
    \[
    q < \Unif{x}{y}(a,b) \text{ if }
    q(y-x) < \Leb{x}{y}(a,b)
    \text{ or } (a < x \leq y < b \text{ and } q<1)
    \text{.}
    \]
\end{definition}

\begin{proposition} \label{prop:uniform}
  $\Unif{x}{y}$ is a probability valuation, and
  \[
    (y-x) \Unif{x}{y} = \Leb{x}{y}
    \text{.}
  \]
\end{proposition}
\begin{proof}
  Using Theorem~\ref{thm:cases} it suffices to check the cases $x<y$ and $x=y$.
  
  If $x<y$ then the second alternative in the definition is redundant, and we just get $\Unif{x}{y}=\Leb{x}{y}/(y-x)$, which is clearly a probability valuation.

  If $x=y$ then we get the unique probability valuation on $\{x\}\cong 1$.
\end{proof}

\section{The Fundamental Theorem of Calculus}
\label{sec:FTC}
The FTC has two parts, essentially to show that integration is the inverse of differentiation.

The first part (Theorem~\ref{thm:FTC1}) shows that if $f(x)$ is defined as an integral $\int_{x_0}^{x} g(t)dt$ then $f$ is differentiable, with derivative $g$.

The second part (Theorem~\ref{thm:FTC2}) shows that if $f$ is differentiable, then it can be recovered (up to constant) as the integral of its derivative.

\subsection{First Fundamental Theorem of Calculus}
\label{sec:FTC1}
Suppose $f(x)$ is defined as an integral $\int_{x_0}^{x} g(t)dt$.
The slope map \emph{for the case} $x\ne y$ is
\[
  \slope{f}(x,y) = \frac{\int_{x}^{y}g(t)dt}{y-x}
  \text{,}
\]
and our task geometrically is to find a uniform definition that includes the case $x=y$.
To do this we reinterpret the RHS of the above equation in a crucial way, as $\int_{[x,y]} gd\Unif{x}{y}$,
where $\Unif{x}{y}$ is the uniform measure.
By Proposition~\ref{prop:uniform} this is the same value as before if $x\ne y$; but if $x=y$ the uniform measure (on the one-point space $\{x\}\cong 1$) is still well defined, and the integral evaluates as $g(x)$.
Since $\Unif{x}{y}$ is defined \emph{geometrically} from $x$ and $y$, we have our required $\slope{f}(x,y)$, and have proved differentiability.

\begin{theorem} \label{thm:FTC1}
	Let $f$ be defined on some real open interval by
	\[
      f(x) = \int_{x_0}^x g(t)dt
	  \text{.}
	\]
	Then $f$ is differentiable, with derivative $f'=g$.
\end{theorem}
\begin{proof}
	Define
	\[
	  \slope{f}(x,y)
	    =\int_{[x',y']} g d\Unif{x'}{y'}
	\]
	where $x'=\min(x,y)$ and $y'=\max(x,y)$.
	We wish to show that
	\[
	  f(y)-f(x) = \slope{f}(x,y)(y-x)
	  \text{.}
	\]
	The collection of pairs $(x,y)$ for which this holds is a subspace, so it suffices to check the cases $x\ne y$ and the closed complement $x=y$.
	The case $x=y$ is obvious, as both sides of the equation are 0.
	
	If $x\ne y$ then either $x<y$ or $y<x$.
	In the first case,
	\[ \begin{split}
	  \slope{f}(x,y) (y-x)
	    &=\int_{[x,y]} g d\Unif{x}{y}(y-x)
	    = \int_{[x,y]} g d \Leb{x}{y}
	    = \int_x^y g(t) dt
	  \\
	    &= \int_{x_0}^y g(t)dt - \int_{x_0}^x g(t)dt
	    = f(y)-f(x)
	  \text{.}
	\end{split} \]
    The second case, $y<x$, is similar but negated.
    
    This shows that $f$ is differentiable.
    its derivative is
    \[
      f'(x) = \slope{f}(x,x)
        =\int_{[x,x]} g d\Unif{x}{x} = g(x)
      \text{.}
    \]
\end{proof}

\subsection{Second Fundamental Theorem of Calculus} \label{sec:FTC2}

\begin{theorem} \label{thm:FTC2}
  Let $f\colon U \to \Reals$ be differentiable, defined on an open interval $U$ of $\Reals$.
  Then, for all $x,y$ in $U$,
  \[
    f(y)-f(x) = \int_x^y f'(u) du
    \text{.}
  \]
\end{theorem}
\begin{proof}
  Fixing $x_0\oftype U$, let us define
  \[
    g_{x_0}(x) = \int_{x_0}^x f'(t)dt
    \text{.}
  \]

  By the first part, we know that $g_{x_0}$ is differentiable, with derivative $g'_{x_0}=f'$,
  and it follows that $(f-g_{x_0})' = 0$.
  Then as a corollary to Rolle's Theorem~\cite{CViet} it follows that $f-g_{x_0}$ is constant on $U$,
  the constant being $f(x_0)-g_{x_0}(x_0)=f(x_0)$.
  Hence $f(x)-g_{x_0}(x)=f(x_0)$ for all $x_0$ and $x$, which gives us the required result. 
\end{proof}

\section{Differentiating powers} \label{sec:DiffPower}

We first address the case of $x^\alpha$, where $\alpha\oftype\Reals$ is fixed, and $x>0$.

\begin{theorem} \label{thm:powerdiff}
  If $\alpha\oftype\Reals$, then the map $x\mapsto x^\alpha$ is differentiable, with derivative $x \mapsto \alpha x^{\alpha-1}$.
\end{theorem}
\begin{proof}
  We first cover three special cases for rational $\alpha$, relating to a natural number $n$.
  
  If $\alpha=n$, then we can calculate
  \[
    \slope{\left( x \mapsto x^n \right)}(x,y)
      = \sum_{i=0}^{n-1} y^{n-1-i}x^i
    \text{,}
  \]
  so the derivative is $nx^{n-1} = \alpha x^{\alpha-1}$.
  
  If $n\geq 1$, then $\slope{\left( x \mapsto x^n \right)}(x,y)$ is positive for all $x,y$ (which are assumed positive).
  It follows from Lemma~\ref{lem:invdiff} that its inverse
  $x \mapsto x^{1/n}$ is differentiable, with derivative
  \[
    \frac{1}{n \left( x^{1/n)}\right)^{n-1}}
      = \frac{1}{n} x^{\frac{1}{n}-1}
    \text{.}
  \]
  This covers the case $\alpha = 1/n$.
  
  For negative integers $\alpha = -n$ we have
  \[
    y^{-n}-x^{-n} = -\frac{y^n-x^n}{x^n y^n}
      = -(y-x)\frac{\sum_{i=0}^{n-1} y^{n-1-i}x^i}{x^n y^n}
    \text{,}
  \]
  so
  \[
    \slope{\left( x \mapsto x^{-n} \right)}(x,y)
      = -\frac{\sum_{i=0}^{n-1} y^{n-1-i}x^i}{x^n y^n}
  \]
  and the derivative is $(-n)x^{n-1-2n} = \alpha x^{\alpha-1}$.
  
  We now observe that if $\alpha$ and $\beta$ both have the property expressed in the statement, then so does $\alpha\beta$.
  By the chain rule  the map $x\mapsto x^{\alpha\beta} = \left( x^\alpha\right)^\beta$ is differentiable, with derivative at $x$ equal to
  \[
    \beta \left( x^\alpha \right)^{\beta-1} \alpha x^{\alpha-1}
      = \alpha\beta x^{\alpha\beta - 1}
    \text{.}
  \]
  
  It follows, based on the special cases, that $\alpha$ has the property of the statement for all \emph{rational} $\alpha$.
  Hence for rational $\alpha$, by the FTC (Theorem~\ref{thm:FTC2}),
  \[
    x^\alpha = 1 + \alpha\int_{1}^{x} t^{\alpha-1} dt
    \text{.} 
  \]
  
  Fixing $x$, the two sides of the above equation define two maps on $\alpha$, and their equalizer, a closed subspace of $\Reals$, includes $\Rat$ and hence is the whole of $\Reals$.
  We can now apply the FTC again to conclude the statement of the theorem for arbitrary $\alpha$.
\end{proof}

\section{Differentiating logarithms and exponentials} \label{sec:DiffLogExp}

First, we define the natural logarithm $\ln x$ as the integral $\int_1^x \frac{dt}{t}$.
Then some familiar manipulations show that $\log_\gamma x = \ln x/\ln \gamma$, so $\log_\gamma$ can be got as an integral.
Now (Theorem~\ref{thm:logdiff}) the Fundamental Theorem of Calculus (FTC) tells us that $\log_\gamma x$ is differentiable, with derivative $((\ln \gamma)x)^{-1}$.
Let us write $\exp_{\gamma}$ for the map $x \mapsto \gamma^x$.
Since (at least for $\gamma \ne 1$) it is the functional inverse of $\log_\gamma$,
an argument (Theorem~\ref{thm:expdiff}) involving
the chain rule tells us that it too is differentiable, with derivative $(\ln \gamma) \exp_{\gamma}$.

\begin{definition}
\label{def:ln}
  If $0 < x$ then the \emph{natural logarithm} $\ln x$ is defined as
  \[
    \ln x = \int_1^x \frac{dt}{t}
    \text{.}
  \]
\end{definition}
  
$\ln$ is strictly increasing, $\ln 1 = 0$, and $\ln 4 > \frac 1 2 + \frac 1 3 + \frac 1 4 > 1$, so we can define \emph{Euler's number,} $e$, as the unique positive real such that $\ln e = 1$.
  
\begin{proposition} \label{prop:logEq}
  Let $f$ be a monoid homomorphism from $(0,\infty)$ under multiplication to $\Reals$ under addition.
  Then $f(\gamma^x) = xf(\gamma)$ for all $x\oftype\Reals$ and $\gamma\oftype(0,\infty)$.
  
  If $\gamma \ne 1$ then $f(x) = f(\gamma)\log_\gamma x$.
\end{proposition}
\begin{proof}
  Once we have the first equation, the second follows from the fact that $\exp_{\gamma}$ is inverse (as map) to $\log_\gamma$. It remains to prove the first.
  
  For natural numbers $x$ we use induction, and then for non-negative rationals $x=a/b$ we have
  \[
    f(\gamma^{\frac{a}{b}})
      = a f(\gamma^{\frac{1}{b}})
      = \frac a b \left(b f(\gamma^{\frac{1}{b}}) \right)
      = \frac a b f(\gamma)
    \text{.}
  \]
  To extend this to negative rationals, we have
  $0 = f(\gamma^x \gamma^{-x}) = f(\gamma^x)+f(\gamma^{-x})$,
  so $f(\gamma^{-x}) = -f(\gamma^x) = -x f(\gamma)$.
  
  Now, fixing $\gamma$, we have two maps $f_1(x)=xf(\gamma)$ and $f_2(x) = f(\gamma^x)$ that agree on all rationals.
  Their equalizer is a closed subspace of $\Reals$, and by the density of the rationals it is the whole of $\Reals$.
\end{proof}

\begin{lemma}\label{lem:lnHom}
	$\ln$ is a homomorphism in the sense of Proposition~\ref{prop:logEq}.
\end{lemma}
\begin{proof}
	Straight away, $\ln 1=0$. Then, by substituting $t'=t/\gamma$ and using equation~\eqref{eq:intsum}, we have
	\[
	  \int_\gamma^{\gamma\gamma'}\frac{dt}{t}
	    = \int_1^{\gamma'} \frac{dt'}{t'} = \ln\gamma'
	    \text{,}
	\]
	and
	\[
	  \ln(\gamma\gamma')
	    = \int_1^\gamma \frac{dt}{t} + \int_\gamma^{\gamma\gamma'}\frac{dt}{t}
	    = \ln\gamma + \ln\gamma'
	    \text{.}
	\]
\end{proof}

\begin{corollary} \label{cor:lnlog}
  $x\ln\gamma = \ln(\gamma^x)$ for all $x\oftype\Reals$ and $\gamma\oftype(0,\infty)$.

  If $\gamma \ne 1$ then $\ln x = \ln\gamma \log_\gamma x$.

  $\ln = \log_e$.
\end{corollary}
Of course, we may also write $\exp x = e^x$.

\begin{theorem} \label{thm:logdiff}
  If $\gamma\ne 1$ then $\log_\gamma$ is differentiable, with derivative $x\mapsto \frac{1}{(\ln\gamma) x}$.
\end{theorem}
\begin{proof}
  By Corollary~\ref{cor:lnlog} we have
  \[
    \log_\gamma x = \int_1^x \frac{dt}{(\ln \gamma) t}
    \text{,}
  \]
  so the result is immediate from the FTC, Theorem~\ref{thm:FTC1}.
\end{proof}

\begin{theorem} \label{thm:expdiff}
  If $\gamma\oftype(0,\infty)$, then $\exp_{\gamma}$ is differentiable, with derivative $(\ln \gamma)\exp_{\gamma}$.
\end{theorem}
\begin{proof}
  First, suppose $\gamma\ne 1$, so that
  the maps $\exp_\gamma$ and $\log_\gamma$ are mutually inverse.
  From Theorem~\ref{thm:logdiff} and the FTC we have that
  \[
    \slope{\log}_\gamma(x,y) = \frac{\int_{[x',y']} \left(t \mapsto \frac{1}{t}\right)d\Unif{x'}{y'}}{\ln\gamma}
    \text{,}
  \]
  where $x'=\min(x,y)$ and $y'=\max(x,y)$.
  This is non-zero, because the integrand is positive and $\ln\gamma$ is non-zero,
  so by Lemma~\ref{lem:invdiff} it follows that $\exp_\gamma$ is differentiable and $\exp'_{\gamma}(x) = (\ln \gamma) \exp'_{\gamma}(x)$.

  By the FTC, it follows that, for $\gamma\ne 1$,
  \[
    \gamma^x = 1+\ln \gamma \int_0^x \gamma^t dt
    \text{.}
  \]
  Let us fix $x$ and look at the space of those $\gamma$s for which the above equation holds, a (closed) subspace of $\Reals$. It contains the open subspace of all $\gamma\ne 1$; but trivially also contains its closed complement, for $\gamma=1$: hence it is the whole of $\Reals$.
  Applying the FTC again, we get the result for all $\gamma$.
\end{proof}

\section{Conclusions} \label{sec:Conc}

The broad message of this work is that integration is more fundamental than differentiation, providing a vital tool for defining the slope maps involved in differentiation.

For example, for the slope map for $\gamma^x$ we can quickly reduce to the question of defining the map $(\gamma^x-1)/x$, continuously at 0.
It is not obvious how one might do that geometrically. Our FTC sidesteps the problem by using an integral with respect to the uniform valuation.

Similarly, without integrals, we do not have an explicit expression for the slope map of $x^\alpha$ when $\alpha$ is irrational.

In their need for this, the slope maps $\slope{f}$ differ from the derivatives $f'$, which can be calculated in a compositional way from a definition of $f$.
In fact, this is what makes an effective differential \emph{calculus.}

The point-free account~\cite{Integration} of 1-sided integration, guided strongly by the recognition of 1-sided reals, is simple, but nonetheless very general.
Clearly it is assisted by two simplifying assumptions.

One is that all spaces are topological, and all measures are regular (valuations). For many practical purposes (as here), that does not seem much of a limitation, and it avoids many of the intricacies and foundational problems of measure theory. 

The second is the fact that all maps are continuous, which will appear restrictive to a mathematician happy to integrate discontinuous functions.
However, these are still accessible if one varies the spaces to give different topologies on classically equivalent sets of points.

For example, consider $H$, the unit step function at 0 on the reals. This can be decomposed as 
$\overrightarrow{H}\colon \overrightarrow{[-\infty, \infty]} \to \overrightarrow{[-\infty, \infty]}$, with $\overrightarrow{H}(0)=0$,
and $\overleftarrow{H}\colon \overleftarrow{[-\infty, \infty]} \to \overleftarrow{[-\infty, \infty]}$, with $\overleftarrow{H}(0)=1$.
Putting these together gives a map $H$ from $\Reals$ to
$\overrightarrow{[-\infty, \infty]} \times \overleftarrow{[-\infty, \infty]}$ whose image is always disjoint, but fails to be located at $H(0)$.
It seems plausible that out of these we can get a 2-sided integral with respect to Lebesgue valuations $\Leb{x}{y}$.

Even under these assumptions, there is still much work to be done in enabling point-free integration to take on the full role of classical integration. One huge gap is that of integration in differential manifolds.

\bibliographystyle{amsalpha}
\bibliography{MyBiblio}
\end{document}